\documentclass[a4paper, english]{amsart}
\usepackage{amsfonts}
\usepackage{booktabs}

\usepackage{amsmath}
\usepackage{amssymb}
\usepackage[small,nohug,heads=LaTeX]{diagrams}
\usepackage{pstricks}
\usepackage{verbatim}
\usepackage{enumerate}
\usepackage{mathtools}

\newtheorem{conj}{Conjecture}

\newtheorem{thm}{Theorem}[section]
\newtheorem{cor}[thm]{Corollary}
\newtheorem{lem}[thm]{Lemma}
\newtheorem{prop}[thm]{Proposition}
\theoremstyle{definition}

\theoremstyle{remark}

\numberwithin{equation}{section}

\newcommand{\bC}{\mathbb{C}}

\newcommand{\bP}{\mathbb{P}}
\newcommand{\bQ}{\mathbb{Q}}

\newcommand{\bZ}{\mathbb{Z}}

\newcommand\lra{\longrightarrow}

\newcommand\trf{\mathrm{trf}}

\newarrow {Onto} ----{>>}
\newarrow {Equals} ===={=}

\newcommand{\CircNum}[1]{\ooalign{\hfil\raise .00ex\hbox{\scriptsize #1}\hfil\crcr\mathhexbox20D}}

\newcommand{\Sym}{\mathrm{Sym}}

\mathchardef\ordinarycolon\mathcode`\:
\mathcode`\:=\string"8000
\begingroup \catcode`\:=\active
  \gdef:{\mathrel{\mathop\ordinarycolon}}
\endgroup

\title[Cohomology of $\mathrm{Aut}(F_n)$]{The stable cohomology of automorphisms of free groups with coefficients in the homology representation}

\author{Oscar Randal-Williams}
\thanks{The author was supported by ERC Advanced Grant No.\ 228082, and the Danish National Research Foundation through the Centre for Symmetry and Deformation.}
\address{Institut for Matematiske Fag\\
Universitetsparken 5\\
DK-2100 K{\o}benhavn {\O}\\
Denmark
}
\email{o.randal-williams@math.ku.dk}

\begin{document}

\begin{abstract}
We study the cohomology of $\mathrm{Aut}(F_n)$ and $\mathrm{Out}(F_n)$ with coefficients in the modules $\wedge^q H$, $\wedge H^*$, $\Sym^q H$ or $\Sym^q H^*$, where $H$ is the $\mathrm{Out}(F_n)$-module obtained by abelianising the free group $F_n$. For reasons which are not conceptually clear, taking coefficients in $H$ and its related modules behaves in a far less trivial way than taking coefficients in $H^*$ and its related modules. Based on a conjectural homology stability theorem for spaces of graphs labeled by a simply connected background space, we give a stable integral calculation of these groups in low degrees, and modulo a further conjecture a stable rational calculation in all degrees.
\end{abstract}
\maketitle

\section{An observation regarding homology stability}

Recently, Galatius \cite{galatius-2006} has proved the remarkable theorem that the two natural homomorphisms
$$\Sigma_n \lra \mathrm{Aut}(F_n) \lra \mathrm{Out}(F_n)$$
induce homology isomorphisms in degrees $2* \leq n - 4$ with integral coefficients. His approach is to model $B\mathrm{Out}(F_n)$ as the space $\mathcal{G}_n$ of graphs of the homotopy type of $\vee^n S^1$, and $B\mathrm{Aut}(F_n)$ as the space $\mathcal{G}^1_n$ of pointed graphs of the same homotopy type. He then produces a natural map from such spaces of graphs to the infinite loop space $Q_0(S^0)$, which he shows has a certain homological connectivity.

At the same time, Satoh \cite{Satoh, Satoh2} has studied the low dimensional homology of $\mathrm{Aut}(F_n)$ and $\mathrm{Out}(F_n)$ with coefficients in the module $H$ given by the abelianisation of $F_n$, and cohomology with coefficients in the dual module $H^*$. His methods are those of combinatorial group theory, and proceed by calculation with a presentation of these groups. From now on, we write $H_\bQ$ and $H_\bQ^*$ for the rationalised $\mathrm{Out}(F_n)$-modules.

The groups $\mathrm{Aut}(F_n)$ and $\mathrm{Out}(F_n)$ fit into a more general family of groups denoted $\Gamma_{n,s}$ by Hatcher--Vogtmann \cite{HV}, where $\mathrm{Aut}(F_n) = \Gamma_{n, 1}$ and $\mathrm{Out}(F_n) = \Gamma_{n, 0}$. Classifying spaces for these may be taken to be the spaces $\mathcal{G}_n^s$ of graphs of the homotopy type of $\vee^n S^1$ equipped with $s$ distinct ordered marked points.

Hatcher and Vogtmann \cite{HV} prove that the map $\mathcal{G}_n^s \to \mathcal{G}_{n+1}^{s}$ (defined for $s > 0$) that adds a loop induces an integral homology isomorphism in degrees $2* \leq n-2$ (and induces a rational homology isomorphism in degrees $5* \leq 4n-10$). Furthermore, the map $\mathcal{G}_n^s \to \mathcal{G}_n^{s-1}$ that forgets a marked point is an integral homology isomorphism in degrees $2* \leq n-3$ (or $2* \leq n-4$ if it is the last marked point).

We can draw an immediate observation regarding homology with coefficients in $H$ from the homology stability theorem. There is an extension $F_n \to \mathrm{Aut}(F_n) \to \mathrm{Out}(F_n)$, and the corresponding Leray--Hochschild--Serre spectral sequence has two rows. However, as the projection is a homology equivalence in a range of degrees we deduce
\begin{prop}
The groups $H^*(\mathrm{Out}(F_n);H^*)$ are zero for $2* \leq n-6$.
\end{prop}

Similarly, there is a fibration with section $\vee^n S^1 \to \mathcal{G}_n^2 \to \mathcal{G}_n^1$ and the projection map is a homology equivalence in a range of degrees, so
\begin{prop}\label{prop:AutFnHomologyRep}
The groups $H^*(\mathrm{Aut}(F_n);H^*)$ are zero for $2* \leq n-4$.
\end{prop}

More generally, the map $\mathcal{G}_n^{k+1} \to \mathcal{G}_n^1$ has fibre $(\vee^n S^1)^k$ and $\mathrm{Aut}(F_n)$ acts on its homology diagonally. Thus for $k=2$ the Serre spectral sequence has three rows, with
$$E_2^{*, 1} = H^*(\mathrm{Aut}(F_n); H^* \oplus H^*) = 0 \quad \text{for $2* \leq n-4$ by Proposition \ref{prop:AutFnHomologyRep}}$$
and
$$E_2^{*, 2} = H^*(\mathrm{Aut}(F_n); H^* \otimes H^*).$$
Using the fact that the projection map is a homology equivalence in degrees $2* \leq n-2$, we deduce that $H^*(\mathrm{Aut}(F_n); H^* \otimes H^*)=0$ for $2* \leq n - 8$. Continuing in this way for higher $k$, we establish the following proposition.
\begin{prop}\label{prop:VanishTensorPowers}
For all $k \geq 1$, $H^*(\mathrm{Aut}(F_n); (H^*)^{\otimes k})$ is zero for $2* \leq n - 4k$.
\end{prop}

We are interested in studying the coefficient systems given by exterior powers $\wedge^q H$ and $\wedge^q H^*$, and symmetric powers $\Sym^q H$ and $\Sym^q H^*$. After rationalising, these modules are summands of $H_\bQ^{\otimes q}$ and $(H^*_\bQ)^{\otimes q}$ respectively, from which we deduce the following corollary.
\begin{cor}\label{cor:VanishingExtPowers}
For $q \geq 1$, the group $H^*(\mathrm{Aut}(F_n); \wedge^qH^*_\bQ)$ is zero for $2* \leq n - 4q$, as is $H^*(\mathrm{Aut}(F_n); \Sym^qH^*_\bQ)$.
\end{cor}

These observations are no doubt known to the experts. The purpose of this note is to outline several related calculations one can make, modulo certain conjectures, about the ``dual" calculation, i.e.\ that of the groups $H^*(\mathrm{Aut}(F_n); \wedge^qH_\bQ)$ and $H^*(\mathrm{Aut}(F_n); \Sym^qH_\bQ)$. We perform all calculations in cohomology; the dual statements about homology may be recovered via the universal coefficients theorem for coefficient modules.

\section{Homology stability with coefficient systems}

By Galatius' theorem, the groups $\mathrm{Aut}(F_n)$ are closely related to the symmetric groups, but also share many properties with mapping class groups of surfaces. These three families of groups are known to exhibit homological stability for integral homology, but symmetric groups and mapping class groups also exhibit homological stability for certain systems of coefficients, those of ``finite degree", a notion that is originally due to Dwyer \cite{DwyerHomStab} in his study of homological stability for general linear groups with coefficient systems.

\begin{conj}\label{conj:StabilityWithCoefficients}
If $V$ is a coefficient system of degree $\leq k$, then $H_*(\mathrm{Aut}(F_n);V) \to H_*(\mathrm{Aut}(F_{n+1});V)$ induces an isomorphism in degrees $2* \leq n-k-2$.
\end{conj}

This conjecture is also suggested by Proposition \ref{prop:VanishTensorPowers} and Corollary \ref{cor:VanishingExtPowers}. As $\wedge^q H$ is a coefficient system of degree $q$, this conjecture implies an improvement of Corollary \ref{cor:VanishingExtPowers}: that $H_*(\mathrm{Aut}(F_n); \wedge^q H_\bQ)$ is zero in degrees $2* \leq n-q-2$.

\section{Graphs labelled by a space $X$}

Let us denote by $\mathcal{G}_n(X)$ a suitably topologised (cp.\ \cite[\S 2]{galatius-2006}) space of graphs of the homotopy type of $\vee^n S^1$ equipped with a continuous map to a space $X$, and $\mathcal{G}^1_n(X)$ the analogue with pointed graphs equipped with pointed  maps to a pointed space $X$. The following may be proved essentially following \cite{galatius-2006}, but using ideas from \cite{GR-W}.

\begin{thm}
There is a natural map
$$\mathcal{G}^1_\infty(X) \lra Q_0(X_+)$$
to the free infinite loop space on the pointed space $X_+$, which is an integral homology equivalence as long as $X$ is path-connected.
\end{thm}

In order to obtain information about $\mathcal{G}_n(X)$ for finite $n$, we assume for the remainder of this note the following conjecture regarding a homological stability phenomenon for the spaces $\mathcal{G}_n(X)$.

\begin{conj}\label{conj}
Suppose $X$ is simply connected. Then the map
$$\mathcal{G}^1_n(X) \lra \mathcal{G}^1_{n+1}(X)$$
induces a homology isomorphism in degrees $2* \leq n - 3$, and the map
$$\mathcal{G}^1_n(X) \lra \mathcal{G}_{n}(X)$$
induces a homology isomorphism in degrees $2* \leq n - 5$.
\end{conj}

Such a conjecture is of course motivated by the analogous theorem for spaces of surfaces, and would surely be unsurprising to the experts. We believe there should be no essential difficulty in establishing Conjecture \ref{conj}, proceeding analogously to the case of surfaces. Our preferred approach in this case is of course \cite{R-WResolution}, but it would also follow from Conjecture \ref{conj:StabilityWithCoefficients} by the methods of \cite{CM}.

\section{Integral calculations}

We first apply the above discussion to the case $X = \bC\bP^\infty = K(\bZ,2)$. In this case $\mathcal{G}_n^1(\bC\bP^\infty)$ classifies families of pointed graphs with a complex line bundle on the total space, and a trivialisation at the marked point. There is a homotopy fibration
$$BH^* \simeq \mathrm{map}_*(\vee^n S^1, \bC\bP^\infty) \lra \mathcal{G}_n^1(\bC\bP^\infty) \lra  \mathcal{G}_n^1(*) \simeq B\mathrm{Aut}(F_n)$$
with section (given by taking the constant map to the basepoint) and an associated Serre spectral sequence
\begin{equation}\label{eq:sseq}
E_2^{p,q} := H^p(\mathrm{Aut}(F_n); \wedge^q H) \Longrightarrow H^{p+q}(\mathcal{G}_n^1(\bC\bP^\infty);\bZ).
\end{equation}
The first column of this spectral sequence is determined by the following lemma.
\begin{lem}\label{lem:VanishingFirstColumn}
The group of invariants $(\wedge^*H)^{\mathrm{Aut}(F_n)}$ is $\bZ$ in degree zero.
\end{lem}
\begin{proof}
$\mathrm{Aut}(F_n)$ acts on $H$ via $GL_n(\bZ)$, so we must show $(\wedge^*H)^{GL_n(\bZ)}$ is trivial in positive degrees, but this is classical.
\end{proof}

Assuming that Conjecture \ref{conj} holds we have a description of the $E_\infty$-page of the spectral sequence (\ref{eq:sseq}), by Galatius' theorem we have a description of the $q=0$ line, and by Lemma \ref{lem:VanishingFirstColumn} we have a description of the $p=0$ line. Hence we are able to make the following calculation.

\begin{prop}[assuming Conjecture \ref{conj}]
We have
\begin{eqnarray*}
& H^1(\mathrm{Aut}(F_n);H) = \bZ & \text{for $n \geq 7$}\\
& H^2(\mathrm{Aut}(F_n);H) = 0 & \text{for $n \geq 9$.}
\end{eqnarray*}
\end{prop}
Of course, the first isomorphism follows from the work of Satoh \cite[Theorem 1]{Satoh} and the universal coefficient theorem, and in fact he establishes it for $n \geq 3$. On the other hand, in \cite{Satoh2} he shows\footnote{Rather, he shows that $H_2(\mathrm{Aut}(F_n);H^*)\otimes \bZ[\tfrac{1}{2}]$ is zero for $n \geq 6$, but the above then follows by universal coefficients.} that $H^2(\mathrm{Aut}(F_n);H) \otimes \bZ[\tfrac{1}{2}]$ is zero as long as $n \geq 6$, but is unable to deal with the $2$-torsion. In this sense our calculation is stronger, at the expense of increasing $n$ from $6$ to $9$.
\begin{proof}
We have the following known cohomology groups in low degrees
\begin{center}
\begin{tabular}{lcccc}
\toprule
$i$ & $0$ & $1$ & $2$ & $3$  \\ \toprule
%$H_*(Q_0(S^0);\bZ)$ & $\bZ$ & $\bZ/2$ & $\bZ/2$ & $(\bZ/2)^2 \oplus \bZ/12$  \\
%$H_*(Q_0(\bC\bP^\infty_+);\bZ)$ & $\bZ$ & $\bZ/2$ & $\bZ \oplus \bZ/2$ & $-$  \\
$H^*(Q_0(S^0);\bZ)$ & $\bZ$ & $0$ &  $\bZ/2$ & $\bZ/2$ \\ 
$H^*(Q_0(\bC\bP^\infty_+);\bZ)$ & $\bZ$ & $0$ &  $\bZ \oplus \bZ/2$ & $\bZ/2$ \\ 
\bottomrule
\end{tabular}
\end{center}
\noindent and hence in total degree 2 the spectral sequence gives a short exact sequence
$$0 \lra \bZ/2 \lra \bZ \oplus \bZ/2 \lra H^1(\mathrm{Aut}(F_n);H) \lra 0$$
as long as $n \geq 7$, and so $H^1(\mathrm{Aut}(F_n);H) \cong \bZ$ in this range. On the other hand, if we consider total degree 3, we see that $E_2^{2,1}$ can support no differentials (as the fibration has a  section), but $E_\infty^{2,1}$ must be zero as long as $n \geq 9$, hence $H^2(\mathrm{Aut}(F_n);H) = 0$.
\end{proof}

This should be contrasted with with a theorem of Bridson and Vogtmann \cite[Theorem B]{BridVogt}, who show that the extension
$$H = F_n / F_n' \lra \mathrm{Aut}(F_n)/F_n' \lra \mathrm{Out}(F_n)$$
is non-trivial for all $n \geq 2$, and hence gives a non-trivial class $\zeta \in H^2(\mathrm{Out}(F_n);H)$. In particular, we see that $H^2(\mathrm{Out}(F_n);H) \to H^2(\mathrm{Aut}(F_n);H)$ is never an isomorphism for $n \geq 9$, and hence that although $\mathrm{Aut}(F_n) \to \mathrm{Out}(F_n)$ has homology stability for constant coefficients, it does not (in general) for coefficients system of finite degree. To study this situation from our point of view, we consider the diagram
\begin{diagram}
BH^*\\
\dTo\\
\mathrm{map}(\vee^n S^1, \bC\bP^\infty) & \rTo & \mathcal{G}_n(\bC\bP^\infty) & \rTo^{\pi} \mathcal{G}_n(*) \simeq B\mathrm{Out}(F_n)\\
\dTo^p\\
\bC\bP^\infty
\end{diagram}
where the row and column are fibrations. Note that $p$ is a trivial fibration and is split via the inclusion $s : \bC\bP^\infty \to \mathrm{map}(\vee^n S^1, \bC\bP^\infty)$ of the constant maps, and there is an inclusion $\iota : \mathcal{G}_n(*) \times \bC\bP^\infty \to \mathcal{G}_n(\bC\bP^\infty)$ of the graphs with constant maps to $\bC\bP^\infty$. The Leray--Serre spectral sequence for the horizontal fibration is
\begin{equation}\label{eq:sseqOut}
\bar{E}_2^{p,*} := H^p(\mathrm{Out}(F_n); \wedge^* H) \otimes \bZ[a] \Longrightarrow H^*(\mathcal{G}_n(\bC\bP^\infty);\bZ)
\end{equation}
where $a$ is the canonical class in $H^2(\bC\bP^\infty;\bZ)$, so has bidegree $(p,q)=(0,2)$.

\begin{prop}[assuming Conjecture \ref{conj}]\label{prop:IntCalcOut}
We have
\begin{align*}
H^1(\mathrm{Out}(F_n);H) &= 0 \quad\quad\quad\quad \text{\,\,\,\,for $n \geq 7$}\\
H^2(\mathrm{Out}(F_n);H) &= \bZ/(n-1) \quad \text{for $n \geq 9$.}
\end{align*}
\end{prop}
\begin{proof}
We first claim that the map
$$\bC\bP^\infty \overset{s}\lra \mathrm{map}(\vee^n S^1, \bC\bP^\infty) \lra \mathcal{G}_n(\bC\bP^\infty)$$
has image $(n-1)\bZ \subset \bZ = H^2(\bC\bP^\infty;\bZ)$ on second cohomology. By Conjecture \ref{conj}, it is enough to prove this after composing with the map to $Q_0(\bC\bP^\infty_+)$ as long as $n \geq 7$. Up to translation of components, this map is given by the Becker--Gottlieb transfer $\bC\bP^\infty  \to Q_{1-n}(\bC\bP^\infty \times (\vee^n S^1)_+)$ for the trivial graph bundle over $\bC\bP^\infty$ composed with projection to $Q_{1-n}(\bC\bP^\infty_+)$. By standard properties of the transfer, this is $(1-n)$ times the standard inclusion, which on second cohomology induces multiplication by $(1-n)$, as required.

This describes the edge homomorphism of the spectral sequence (\ref{eq:sseqOut}). As it converges to zero for positive Leray filtration in total degree $3$, the differential $d_2 : \bZ = \bar{E}_2^{0, 2} \to \bar{E}_2^{2,1}$ must be onto (so $\bar{E}_2^{2,1}$ is cyclic) and the kernel is $(n-1)\bZ$, so $\bar{E}_2^{2,1} = H^2(\mathrm{Out}(F_n);H) \cong \bZ/(n-1)$. On the other hand, in total degree $2$ we see $(n-1)\bZ = \bar{E}_\infty^{0,2}$ and $\bZ/2 = \bar{E}_\infty^{2,0}$, and it converges to $\bZ/2 \oplus \bZ$, so observing the direction of the Leray filtration we see that $H^1(\mathrm{Out}(F_n);H) = \bar{E}_2^{1,1} =0$.
\end{proof}

The calculation $H^2(\mathrm{Out}(F_n);H) = \bZ/(n-1)$ along with the result of Bridson and Vogtmann \cite{BridVogt} that their class $\zeta \in H^2(\mathrm{Out}(F_n);H)$ remains non-trivial in the group $H^2(\mathrm{Out}(F_n);H/rH)$ for any $r$ not\footnote{The paper \cite{BridVogt} contains an unfortunate misprint, where they make this statement for those $r$ which \emph{are} coprime to $(n-1)$.} coprime to $(n-1)$ implies that the class $\zeta$ generates $H^2(\mathrm{Out}(F_n);H)$.

\section{Symmetric powers}

We now turn to the modules $\Sym^q H$, for which we let the background space be $K(\bZ,3)$. Then there is an equivalence $\mathrm{map}_*(\vee^n S^1, K(\bZ,3)) \simeq K(H^*, 2)$, and so the cohomology of this space as an $\mathrm{Out}(F_n)$-module is $\Sym^* H$ with grading doubled. The relevant spectral sequences are then, in rational cohomology,
\begin{eqnarray}
E_2^{p,2q} := H^p(\mathrm{Aut}(F_n);\Sym^q H_\bQ) \Longrightarrow H^*(\mathcal{G}_n^1(K(\bZ,3));\bQ) \label{eq:sseqSym}\\
\bar{E}_2^{p,*} := H^p(\mathrm{Out}(F_n);\Sym^* H_\bQ) \otimes \Lambda[a] \Longrightarrow H^*(\mathcal{G}_n(K(\bZ,3));\bQ) \label{eq:sseqSymOut}
\end{eqnarray}
where $a \in H^3(K(\bZ,3);\bQ)$ is the tautological class.

\section{Rational calculations}

In order to obtain information in higher degrees, we pass to rational coefficients. We may then calculate
$$H^*(\mathcal{G}_\infty^1(\bC\bP^\infty);\bQ) \cong H^*(Q_0(\bC\bP^\infty_+);\bQ) \cong \bQ[\chi_2, \chi_4, \chi_6, ...]$$
where $\chi_{2i}$ is the cohomology suspension of the $i$-th power of the Chern class. These classes $\chi_{2i}$ may be defined intrinsically on $\mathcal{G}_n^1(\bC\bP^\infty)$ by applying the Becker--Gottlieb transfer \cite{BG} for the universal family of graphs to powers of the first Chern class of the complex line bundle on the universal family. The spectral sequence (\ref{eq:sseq}) in rational cohomology converges to this algebra in the stable range, and the above calculation shows that the element $\chi_2$ is detected in the group $H^1(\mathrm{Aut}(F_n);H_\bQ)$. We may phrase this as saying that the class $\chi_2 \in H^2(\mathcal{G}_n^1(\bC\bP^\infty);\bQ)$ has Leray filtration precisely 1 with respect to the map $\mathcal{G}_n^1(\bC\bP^\infty) \to \mathcal{G}_n^1(*)$. Assuming Conjecture \ref{conj}, the class $\chi_{2i}$ must be detected in the spectral sequence as long as $n \geq 2i+3$, and we may ask in which Leray filtration it is detected. In order to use this calculation with the spectral sequences (\ref{eq:sseq}) and (\ref{eq:sseqOut}) in reverse, we make the following conjecture.

\begin{conj}\label{conj:Collapse}
The spectral sequences (\ref{eq:sseq}) and (\ref{eq:sseqSym}) rationally collapse in total degrees $2* \leq n-3$ and the spectral sequences (\ref{eq:sseqOut}) and (\ref{eq:sseqSymOut}) rationally collapse in total degrees $2* \leq n-5$.
\end{conj}

This conjecture is not unreasonable for $\mathrm{Aut}(F_n)$ by comparison with the analogous theorem for surfaces with boundary proved by Kawazumi \cite[Theorem 1.C]{KawazumiGoD}. The analogous theorem for closed surfaces is proved in \cite{ERW10} but uses complex geometry and Hodge theory in an important way, which is not available here. Thus for $\mathrm{Out}(F_n)$ it is perhaps more speculative.

\begin{thm}[assuming Conjecture \ref{conj}]\label{thm:LerayFilt}
With respect to the map $\mathcal{G}_n^1(\bC\bP^\infty) \to \mathcal{G}_n^1(*)$, the element $\chi_{2i}$ has Leray filtration precisely $i$, and more generally, the monomial $\chi_{2}^{a_1} \cdot \chi_{4}^{a_n} \cdots \chi_{2n}^{a_n}$ has Leray filtration precisely $a_1 + 2a_2 + \cdots + na_n$. Thus the Leray filtration coincides with the filtration by half the cohomological degree, and the associated graded algebra is $\bQ[x_2, x_4, \ldots]$ where $x_{2i}$ has bidegree $(i,i)$.
\end{thm}

Let us write $\mathcal{S}_{g,1}(\bC\bP^\infty)$ for the moduli space of surfaces of genus $g$ with a single boundary component and a map to $\bC\bP^\infty$, as defined in \cite{CM}. There is a natural map of fibrations
\begin{diagram}
\mathrm{map}_\partial(\Sigma_{g,1}, \bC\bP^\infty) & \rTo^{\simeq} & BH^1(F_n;\bZ)\\
\dTo & & \dTo\\
\mathcal{S}_{g,1}(\bC\bP^\infty) & \rTo &\mathcal{G}_{2g}^1(\bC\bP^\infty) \\
\dTo & & \dTo\\
\mathcal{S}_{g,1}(*) & \rTo &\mathcal{G}_{2g}^1(*)
\end{diagram}
where the lower map essentially sends a surface to its 1-skeleton, and it is easy to see that the class $\chi_{2i}$ is pulled back via the middle map to the class named $\kappa_{1, i}$ in \cite{ERW10}, as they are both defined by transferring powers of the first Chern class. %We show there, using results of Kawazumi \cite{Kawazumi98}, that $\widetilde{M}_{1, i}$ has Leray filtration precisely $i$ with respect to the vertical map, and so by naturality of the Leray filtration $x_i$ has filtration at most $i$ with respect to the vertical map.
The map of Leray--Serre spectral sequences for these two fibrations is then
\begin{diagram}
H^p(\Gamma_{g,1};\wedge^q H_\bQ) & \lTo & H^p(\mathrm{Aut}(F_{2g});\wedge^q H_\bQ)\\
\dImplies & & \dImplies\\
H^{*}(\mathcal{S}_{g,1}(\bC\bP^\infty);\bQ) & \lTo & H^{*}(\mathcal{G}_{2g}^1(\bC\bP^\infty);\bQ)\\
\dEq_{3* \leq 2g-2} & & \dEq_{2* \leq n-3}\\
\bQ[\kappa_{i,j} \vert i+j > 0, j \geq 0, i\geq -1] & \lTo & \bQ[\chi_2, \chi_4, ...]\\
\kappa_{1, j} & \lMapsto & \chi_{2j}
\end{diagram}

\begin{proof}[Proof of Theorem \ref{thm:LerayFilt}]
Kawazumi \cite{KawazumiMagnus} has defined certain cohomology classes $\bar{h}_p \in H^p(\mathrm{Aut}(F_n); H^{\otimes p})$, which under the projection $H^{\otimes p} \to \wedge^p H$ give classes we give the same name, $\bar{h}_p \in H^p(\mathrm{Aut}(F_n); \wedge^p H)$. These lie on the $E_2$-page of a spectral sequence converging to $H^{*}(\mathrm{Aut}(F_n) \ltimes H^*;\bZ)$. Following Kawazumi's construction in \cite{KawazumiMagnus} and his methods in \cite{KawazumiGoD}, we will produce a class in $H^{2p}(\mathrm{Aut}(F_n) \ltimes H^*;\bZ)$ that lifts $\bar{h}_p$, and hence see that $\bar{h}_p$ is a permanent cycle in this spectral sequence. 

Let $\bar{A}_n := \mathrm{Aut}(F_n) \ltimes F_n$ and $F_n \to \bar{A}_n \overset{\pi} \to \mathrm{Aut}(F_n)$ be the defining extension. Kawazumi constructs a cocycle $k_0 : \bar{A}_n \to H$ giving a class $[k_0] \in H^1(\bar{A}_n;H)$, which is necessarily a permanent cycle in the Leray--Hochschild--Serre spectral sequence for the extension $H^* \to \bar{A}_n \ltimes H^* \overset{p}\to \bar{A}_n$. Thus it detects a class $K_0 \in H^2(\bar{A}_n \ltimes H^*;\bZ)$. We define the cohomology class
$$K_0^p \cdot p^*([k_0]) \in H^{2p+1}(\bar{A}_n \ltimes H^*;H)$$
and its image under
$$H^{2p+1}(\bar{A}_n \ltimes H^*;H) \overset{\pi_!}\lra H^{2p}(\mathrm{Aut}(F_n) \ltimes H^*; H \otimes H^*) \overset{\epsilon}\lra H^{2p}(\mathrm{Aut}(F_n) \ltimes H^*;\bZ)$$
is defined to be $\bar{H}_p$. As $K_0$ is detected by $[k_0]$, it follows that $K_0^p \cdot p^*([k_0])$ is detected by $[k_0]^{p+1} \in H^{p+1}(\bar{A}_n;H^{\otimes p+1})$. Under fibre integration and the augmentation this defines of $\bar{h}_p$, as required.

Kawazumi shows that restricted to the mapping class group $\Gamma_{g,1}$ the class $p! \bar{h}_p$ becomes $-m_{1,p} \in H^p(\Gamma_{g,1};\wedge^p H)$ where $m_{i, j}$ are a family of classes he has previously defined in \cite{Kawazumi98}. In \cite{ERW10} Ebert and the author showed that $m_{1, p}$ detects the class $\kappa_{1,p} \in H^{2p}(\mathcal{S}_{g,1}(\bC\bP^\infty))$, and it thus follows that $-p! \bar{h}_p$ detects the class $\chi_{2p}$, and in particular that $\chi_{2p}$ has Leray filtration precisely $p$. A similar argument establishes the theorem for all monomials.
\end{proof}

This proof shows that our classes $\chi_{2p}$ lift Kawazumi's $\bar{h}_p$, and hence that
$$\bQ[\bar{h}_1, \bar{h}_2, \ldots] \lra H^*(\mathrm{Aut}(F_n);\wedge^* H_\bQ) \subset H^*(\mathrm{Aut}(F_n); H_\bQ^{\otimes *})$$
is injective in total degrees $2* \leq n-3$, i.e.\ that these classes are algebraically independent. This was proved by other means in \cite{KawazumiBraid}.

\vspace{2ex}

We now study the same situation for $\mathrm{Out}(F_n)$, where we have the spectral sequence (\ref{eq:sseqOut}) which we may write as
$$H^p(\mathrm{Out}(F_n);\wedge^* H_\bQ) \otimes \bQ[a] \Longrightarrow \bQ[\chi_2, \chi_4, \chi_6, ...] \quad \text{in total degrees $2* \leq n-5$}.$$
The calculation of Proposition \ref{prop:IntCalcOut} shows that $a \in \bar{E}_2^{0,2}$ is a rational permanent cycle, hence detects $\chi_2$ which must then have Leray filtration precisely $0$. We also have the fibration
$$\vee^n S^1 \lra B\mathrm{Aut}(F_n) \ltimes H^* \overset{\pi}\lra B\mathrm{Out}(F_n)\ltimes H^*,$$
which admits a Becker--Gottlieb transfer. Thus $\pi^*$ is injective and so for $i> 1$ let us write $\widetilde{\chi}_{2i} \in H^{2i}(\mathrm{Out}(F_n) \ltimes H^*;\bQ)$ for the unique class having the property that $\pi^*(\widetilde{\chi}_{2i}) = \chi_{2i}$, so $\trf^*_\pi(\chi_{2i}) = \widetilde{\chi}_{2i}$ also. If we write $\tilde{E}_r^{p,q}$ for the spectral sequence of the extension $\mathrm{Out}(F_n) \ltimes H^*$ in rational cohomology, we have an isomorphism $\tilde{E}_\infty^{*,*} \cong \mathcal{G}r\bQ[\widetilde{\chi}_4, \widetilde{\chi}_6, \widetilde{\chi}_8, ...]$ in the stable range, where $\mathcal{G}r$ denotes the associated graded to the Leray filtration. It remains to describe this filtration.

\begin{thm}[assuming Conjectures \ref{conj} and \ref{conj:Collapse}]\label{thm:LerayFiltOut}
With respect to the map $\mathcal{G}_n(\bC\bP^\infty) \to \mathcal{G}_n(*)$, $\chi_2$ has Leray filtration precisely $0$ and for $i > 1$ the element $\chi_{2i}$ has Leray filtration precisely $i$. More generally, the monomial $\chi_{2}^{a_1} \cdot \chi_{4}^{a_n} \cdots \chi_{2n}^{a_n}$ has Leray filtration precisely $2a_2 + \cdots + na_n$.
\end{thm}

We will make use of the following lemma, which follows from Kawazumi \cite[Theorem 7.1]{KawazumiMagnus}. It also follows from general principles: the Becker--Gottlieb transfer with local coefficients \cite[Lemma A.1]{ERW10}.

\begin{lem}
For any $\bZ[\tfrac{1}{n-1}][\mathrm{Out}(F_n)]$-module $M$, the map
$$H^*(\mathrm{Out}(F_n);M) \lra H^*(\mathrm{Aut}(F_n);M)$$
is split injective.
\end{lem}

\begin{proof}[Proof of Theorem \ref{thm:LerayFiltOut}]
As the spectral sequences (\ref{eq:sseq}) and (\ref{eq:sseqOut}) are both assumed to collapse, the filtration on $H^*(\mathrm{Aut}(F_n) \ltimes H^*;\bQ)$ and $H^*(\mathrm{Out}(F_n) \ltimes H^*;\bQ)$ come from the $E_2$-pages. As $H^p(\mathrm{Out}(F_n);\wedge^q H_\bQ) \to H^p(\mathrm{Aut}(F_n);\wedge^q H_\bQ)$ is split injective by the above lemma, it follows that the induced filtration on $H^*(\mathrm{Out}(F_n) \ltimes H^*;\bQ) \hookrightarrow H^*(\mathrm{Aut}(F_n) \ltimes H^*;\bQ)$ agrees with the Leray filtration.
\end{proof}

We now turn to the modules $\Sym^q H$, where we have the rational calculation
$$H^*(\mathcal{G}_\infty^1(K(\bZ,3));\bQ) \cong H^*(Q_0(K(\bZ,3)_+);\bQ) \cong \Lambda[a].$$
Observing that $E_2^{1,2}=H^1(\mathrm{Aut}(F_n);\Sym^1 H_\bQ) = H^1(\mathrm{Aut}(F_n); H_\bQ) = \bQ$, we see that the element $a \in H^3(\mathcal{G}_n^1(K(\bZ,3));\bQ)$ must be detected in $E_\infty^{1,2}$. Thus away from this position the spectral sequence (\ref{eq:sseqSym}) converges rationally to zero. In the spectral sequence (\ref{eq:sseqSymOut}), we see that $a \in E_2^{3,0}$ is a permanent cycle for $n \geq 2$ as in the proof of Proposition \ref{prop:IntCalcOut}. Thus this detects the class $a \in H^3(\mathcal{G}_n(K(\bZ,3));\bQ)$, and we deduce that away from this position the spectral sequence (\ref{eq:sseqSymOut}) converges rationally to zero.

\subsection{A conjectural description of the stable rational cohomology}
Putting these calculations together with Conjecture \ref{conj} leads to the following rather surprising (to the author at least) calculation.
\begin{cor}[assuming Conjectures \ref{conj} and \ref{conj:Collapse}]
In degrees $2* \leq n-2q-3$ we have
$$
H^*(\mathrm{Aut}(F_n); \wedge^q H_\bQ) = 
\begin{cases}
0 & \text{if $* \neq q$}\\
\bQ^{\rho(q)} & \text{if $* = q$}
\end{cases}
$$
where $\rho(q)$ denotes the number of partitions of $q$. In degrees $2* \leq n-2q-5$ we have
$$
H^*(\mathrm{Out}(F_n); \wedge^q H_\bQ) = 
\begin{cases}
0 & \text{if $* \neq q$}\\
\bQ^{\bar{\rho}(q)} & \text{if $* = q$}
\end{cases}
$$
where $\bar{\rho}(q)$ denotes the number of partitions of $q$ into pieces none of which are $1$. In degrees $2* \leq n-4q-3$ we have
$$
H^*(\mathrm{Aut}(F_n); \Sym^q H_\bQ) = 
\begin{cases}
0 & \text{if $q \neq 1$}\\
\bQ & \text{if $q=1$ and $*=1$.}
\end{cases}
$$
In degrees $2* \leq n-4q-5$ we have
$$H^*(\mathrm{Out}(F_n); \Sym^q H_\bQ) = 0.$$
\end{cor}

\bibliographystyle{plain}
\bibliography{MainBib}

\end{document}